\documentclass{amsart}
\usepackage{amsfonts,amssymb,amsmath,amsthm}
\usepackage{url}
\usepackage{enumerate}

\newtheorem{theorem}{Theorem}[section]

\newtheorem{proposition}[theorem]{Proposition}
\newtheorem{corollary}[theorem]{Corollary}
\theoremstyle{definition}
\newtheorem{definition}[theorem]{Definition}
\newtheorem{remark}[theorem]{Remark}
\numberwithin{equation}{section}

\author{U. Bekbaev}
\address{
Department of Science in Engineering, Faculty of Engineering,\\
IIUM, 50728, KL, Malaysia\\
}
\email{bekbaev@iium.edu.my}
\thanks{This research is supported by the MOHE Malaysia under grant FRGS14-153-0394.}

\keywords{linear partial differential equation; change of variables; equivalence; invariant; fiber bundled space}
\subjclass{Primary 35A99, Secondary  35J15, 35L10, 12H20}

\begin{document}
\begin{center}\small{In the name of Allah, Most gracious, Most Merciful.}\end{center}
\vspace{2cm}
\title[On classification and invariants]{On classification and invariants of second order non-parabolic linear partial differential equations in two variables}

\begin{abstract}
       The paper deals with second order abstract linear partial differential equations (LPDE) over a partial differential field with two commuting differential operators. In terms of usual differential equations the main content can be presented as follows. The classification and invariants problems for second order LPDEs with respect to transformations \[x=x(\xi,\eta),\ y=y(\xi,\eta),\ u=h(x,y)v(\xi,\eta),\] where $x,y$ are independent and $u$ is the dependent variable of the LPDE, are considered. Solutions to these problems are given for different subclasses of non-parabolic LPDE which appear naturally in the equivalence problem of LPDE. A criterion for reducibility of such LPDE to LPDE with constant coefficients is offered as well. \end{abstract}

\maketitle


\section{Introduction}

 Let \[ A(x,y)u_{xx}+ B(x,y)u_{xy}+
C(x,y)u_{yy}+ a(x,y)u_{x}+ b(x,y)u_{y}+c(x,y)u= 0 \] be a second order
linear partial differential equation (LPDE) over the ring \\$
C^{\infty}(\mathbf{R^2}, \frac{\partial}{\partial x},
\frac{\partial}{\partial y})$. Under the transformation of dependent variable
 $ v=h(x,y)^{-1}u$ and independent variables $ \xi = \xi(x,y), \eta = \eta(x,y) $ it
transforms to an analogous (equivalent) one \[ A_{1}(\xi,\eta)v_{\xi\xi}+
B_{1}(\xi,\eta)v_{\xi\eta}+ C_{1}(\xi,\eta) v_{\eta\eta}+
a_{1}(\xi,\eta)v_{\xi}+ b_{1}(\xi,\eta)v_{\eta}+ c_{1}(\xi,\eta)v= 0, \] and
equality $\delta= g^{-1}\partial$- the chain rule is valid, where $\partial $ is column vector with
"coordinates"\ $ \partial^{1}= \partial /\partial x$, $ \partial
^{2}=
\partial /\partial y$, $\delta $ is column vector with
"coordinates" \ $\delta^{1}= \partial /\partial \xi$, $\delta^{2}=
\partial /\partial \eta $, $g$ is matrix $g= (g^i_j)_{i,j=
1,2}$ with $g^1_1= \xi_{x}, g^1_2= \eta_{x}, g^2_1=\xi_{y},
g^2_2= \eta_{y}$ and $u(x,y)=h(x,y)v(\xi(x,y),\eta(x,y))$.

A function $f^{\partial}(\mathbf{x})$ of $x_{1},x_{2},...,x_{6}$ and a
finite number of their partial derivatives is said to be invariant (relative invariant) if \[f^{\delta}(V_1)=f^{\partial}(V) \  (\mbox{respectively,} \  f^{\delta}(V_1)=h(x,y)^k\Delta(x,y)^lf^{\partial}(V)) \] for any
$V= (A(x,y), B(x,y),
C(x,y), a(x,y), b(x,y), c(x,y))$ from the domain of $f^{\partial}(\mathbf{x})$, non-vanishing $h(x,y)$ and change of independent variables $ \xi = \xi(x,y), \eta = \eta(x,y) $, where
\[V_1= (A_{1}(\xi,\eta),
B_{1}(\xi,\eta),
C_{1}(\xi,\eta), a_{1}(\xi,\eta),
b_{1}(\xi,\eta), c_{1}(\xi,\eta)),\] $k,l$ are fixed integers and $\Delta(x,y)=\xi_x\eta_y-\xi_y\eta_x$.
Such invariants are important in the classification theory of
differential equations.  For example usage of a single relative
invariant $D\langle V\rangle= B^2-4AC$, namely the
discriminant, gives us a rough classification of the second order real LPDE (elliptic, hyperbolic and
parabolic differential equations). In general every relative invariant $f^{\partial}(\mathbf{x})$ separates its own domain into two parts:
\[ \{V: f^{\partial}(V)\neq 0\} \ \ \mbox{and} \ \  \{V: f^{\partial}(V)= 0\}\] and never an element of the first set and an element of the second one can be equivalent to the same LPDE.

We will consider an abstract algebraic analog of the above situation. For that we use some notations, terminology and results from the theory of differential algebra which one can find
in \cite{Ko}.

Let $(F, \partial^{1}, \partial^{2})$ be any differential field with the commuting differential operators $\partial^{1}, \partial^{2}.$ Consider
  \[  GL^{\partial}(2,F)= \{ g= (g^i_j)_{i,j=1,2} \in
GL(m,F): \partial^{k}g^i_j=
 \partial^{i}g^k_j\ \mbox{ at} \  i,j,k= 1,2\} \] and
any second order linear partial differential equation
\begin{eqnarray}
A(\partial^1)^{2}u+ B\partial^{1}\partial^{2}u+C(\partial^2)^2u+ a\partial^{1}u+ b\partial^{2}u +cu= 0,
\end{eqnarray} where $A,B,C,a,b,c $ are given elements of $F$ and $u$ is an unknown element of $F$.

By transformations $v=h^{-1}u$,
 $\delta=g^{-1}\partial$,
where  $g\in  GL^{\partial}(2,F),$ $h\in F^*=F\setminus \{0\}$, $\partial $ is column
vector with \ "coordinates"\ $
\partial^{1}, \partial^{2}$ and $\delta $
is column vector with \ "coordinates"\ $\delta^{1},
\delta^{2}$, it is
transformed to an analogous  one
\begin{eqnarray}
A_{1}(\delta^1)^{2}v+ B_{1}\delta^{1}\delta^{2}v+C_{1}(\delta^{2})^{2}v+ a_{1}\delta^1v+ b_1\delta^{2}v+ c_1v= 0.
\end{eqnarray}
Moreover the coefficients of these equations are related in the following way.
\begin{eqnarray}
\left\{ \begin{array}{l}
A_{1}=h(A(g^1_1)^{2}+ Bg^1_1g^2_1+ C(g^2_1)^{2}),   \\
B_{1}=h(2Ag^1_1g^1_2+ B(g^1_1g^2_2+ g^1_2g^2_1)+ 2Cg^2_1g^2_2),  \\
C_{1}=h(A(g^1_2)^{2}+ Bg^1_2g^2_2+ C(g^2_2)^{2}), \\
a_{1}=h(A\partial^{1}g^1_1+ B\partial^{1}g^2_1+ C\partial^{2}g^2_1+
 ag^1_1+ bg^2_1)+ 2A\partial^1hg^1_1+ \\ B(\partial^1hg^2_1+ \partial^2hg^1_1)+ 2C\partial^2hg^2_1,\\
b_{1}=h(A\partial^{1}g^1_2+ B\partial^{1}g^2_2+ C\partial^{2}g^2_2+
ag^1_2+ bg^2_2)+ 2A\partial^1hg^1_2+ \\ B(\partial^1hg^2_2+ \partial^2hg^1_2)+ 2C\partial^2hg^2_2,\\
c_1=A(\partial^1)^2h+ B\partial^1\partial^2h+ C(\partial^2)^2h+ a\partial^1h+ b\partial^2h+ ch.
\end{array} \right.
\end{eqnarray}

\begin{definition} Equations $(1.1)$ and $(1.2)$ are said to be $G=F^*\times GL^{\partial}(2,F)$ -equivalent if there exist $h\in F^*$, $g\in GL^{\partial}(2,F)$ such that $ \delta =g^{-1}\partial $ and system of equalities $(1.3)$ is valid. In such case we write  $V_1=
\tau^{\partial}\langle h, g;V\rangle$, where \\ $V_{1}=
(A_{1},B_{1},C_{1},a_{1},b_{1},c_1),\ V= (A,B,C,a,b,c)$.\end{definition}

For the convenience even in the abstract differential field case we
use
\[ u_x \ \ (u_y, \ \ u_{\xi}, \ \ u_{\eta})\quad
\mbox{to denote} \quad \partial^1u \ \ (\mbox{respectively,}
\ \
\partial^2u, \ \ \delta^1u ,
 \ \ \delta^2u).\]

Let $C=\{a\in F: \partial^1a=\partial^2a= 0\}$ be the constant field for $(F, \partial^1, \partial^2)$ and $\mathbf{x}$ stand for the variable vector $(x_{1},x_{2},...,x_{6})$. The set of all linear partial differential equations (1.1) over $F$ of order $\leq 2$ is nothing than $F^6$. Whenever we speak about topology we mean the $\partial$- differential Zariski topology \cite{Ko}. If $W_e\subset F^6$ is an irreducible closed subset we use $C\langle W_e; \partial\rangle $ to denote the field of $\partial$-differential rational functions on $W_e$ with constant coefficients. We use $C\{\mathbf{x}; \partial\}$ to denote the $\partial$-differential ring of $\partial$-differential polynomials in independent variables $(x_{1},x_{2},...,x_{6})$
 over $C$ and
$C\langle \mathbf{x}; \partial\rangle $ for the field of $\partial$-differential rational functions in $\mathbf{x}$ over $C$.

\begin{definition} A differential rational function
$f^{\partial}\langle  \mathbf{x}\rangle\in C\langle W_e; \partial\rangle $ is called to be a $G$ -
invariant (relative invariant) if the equality
\[ f^{g^{-1}\partial}\langle  \tau^{\partial}\langle h,g;\mathbf{x}\rangle\rangle= f^{\partial}\langle \mathbf{x}\rangle \ \  \ (\mbox{respectively},\  f^{g^{-1}\partial}\langle  \tau^{\partial}\langle h,g;\mathbf{x}\rangle\rangle=h^k\Delta^l f^{\partial}\langle \mathbf{x}\rangle) \]
is valid for any $(h,g)\in G$ whenever $f^{\partial}$ is defined at $\mathbf{x}\in W_e$, where $k,l$ are fixed integers and $\Delta= \det g$.\end{definition}

Let us denote the field of $F^*\times GL^{\partial}(2,F)$ -
invariant $\partial$-differential rational functions over $W_e$ by $C\langle W_e; \partial\rangle^G$.

The main aim of the present paper is the classification of second order non-parabolic LPDE with respect to the above considered action of $F^*\times GL^{\partial}(2,F)$ and description of the corresponding field of $F^*\times GL^{\partial}(2,F)$-invariant differential rational functions over $C$. In this paper it is done for some $G$ -invariant subsets of second order non-parabolic LPDE which appear naturally in their $G$-equivalence problem. A criterion for reducibility of such equations to LPDE with constant coefficients is offered as well.

For the second order generic LPDE the $G$-equivalence and invariants are considered in \cite{B1} without classification and reducibility to LPDE with constant coefficients problems. In particular in this paper all results of \cite{B1}, except for Theorem 3.1, are reobtained by the use of more general method than of \cite{B1} (see case a) considered in Section 3).
For the equivalence and invariants problems of such equations with respect only to the change of independent variables one can see \cite{B2, B3}.

 We would like to note that in spite of the importance of second order LPDE their equivalence and invariants problems have not been explored much.
As to the similar problems for the ordinary LDE they are considered by many authors, see for example \cite{Be, Kor, New, Sir, Ud, W}. For an algebraic approach, by means of differential algebra, to the equivalence and invariance problems of ordinary LDE one can see \cite{B4, B5, B6}.

The construction of the paper is as follows. Section 2 is about the key results which are presented in an appropriate form to investigate the classification and invariants problems for any order LPDE. Section 3 is about applications of Section 2 results to different subsets  of non-parabolic second order LPDE.

\section{Preliminaries}

In this section we use $l,m,n$ for any fixed natural numbers and prove some common results to use them in the next section to solve classification and invariants problems for the second order LPDE.

Let
$(F,\partial )$ stand for a field $F$ with fixed commuting system
of differential operators $\partial^1,\partial^2,. . .,\partial^m$ and $C$ be
its constant field i.e. \[C=\{ a\in F:\partial^1a=\partial^2a=. .
.=\partial^ma=0\},\] where $\partial$ stands for the column with "coordinates"\ $\partial^1,\partial^2,. . .,\partial^m$. Elements of $F^m$ are assumed to be written as row vectors and
\[GL^{\partial}(m,F)= \{ g\in  GL(m,F):
\partial^ig^j_k= \partial^jg^i_k \ \mbox{for} \ i,j,k=1,2,...,m\}.\]

\begin{proposition} If the system of differential
operators $\partial^1,. . .,\partial^m$ is linear independent over
$F$ then the differential operators $\delta^1,. .
.,\delta^m$, where $\delta= g^{-1}\partial, g\in GL(m,F)$, commute
with each other if and only if $  g\in GL^{\partial}(m,F)$.\end{proposition}

\begin{proof} Let $g\in GL(m,F)$,\ $\delta= g^{-1}\partial $. It is
clear that linear independence of $\partial^1,. . .,\partial^m$
implies linear independence of $\delta^1,. . .,\delta^m.$ For any
$i,j=1,2,...,m$ we have $\partial^j=
\sum_{k=1}^mg^j_k\delta^k$, \[\partial^i\partial^j=
\sum_{k=1}^m(\partial^i(g^j_k)\delta^k+
g^j_k\partial^i\delta^k)= \sum_{k=1}^m\partial^i(g^j_k)\delta^k
+ \sum_{k=1}^m\sum_{s=1}^mg^j_kg^i_s\delta^s\delta^k.\] Therefore
due to $\partial^i\partial^j= \partial^j\partial^i $ one has
\begin{eqnarray}\begin{array}{c}
\sum_{k=1}^m\partial^i(g^j_k)\delta^k +
\sum_{k=1}^m\sum_{s=1}^mg^j_kg^i_s\delta^s\delta^k=\\
\sum_{k=1}^m\partial^j(g^i_k)\delta^k +
\sum_{k=1}^m\sum_{s=1}^mg^j_kg^i_s\delta^k\delta^s
\end{array}\end{eqnarray}

If $\delta^k\delta^s = \delta^s\delta^k$ for any $k,s=
1,2,...,m$ then due (2.1) one has
\[\sum_{k=1}^m\partial^i(g^j_k)\delta^k  =
\sum_{k=1}^m\partial^j(g^i_k)\delta^k\]  i.e.\ $\partial^i(g^j_k)=
\partial^j(g^i_k)$ for any  $i,j,k= 1,2,...,m$ because of linear independence of
$\delta^1,. . .,\delta^m$. Thus in this case
 $  g\in GL^{\partial}(m,F)$.

Vice versa, if $\partial^i(g^j_k)=
\partial^j(g^i_k)$ for any  $i,j,k= 1,2,...,m$  then due (2.1) at any $a\in F$ one has
\[\sum_{k=1}^m\sum_{s=1}^mg^j_kg^i_s\delta^s\delta^k a=
\sum_{k=1}^m\sum_{s=1}^mg^j_kg^i_s\delta^k\delta^s a\ \mbox{for any}
\ i,j= 1,2,...,m.\] These equalities can be written in the
following matrix form \[g(\delta^1\delta a,. . .,\delta^m\delta
a)g^t= g(\delta^1\delta a,. . .,\delta^m\delta a)^tg^t,\] where
$t$ means the transpose. Therefore $(\delta^1\delta a,. .
.,\delta^m\delta a)= (\delta^1\delta a,. . .,\delta^m\delta a)^t$
i.e.
 $\delta^k\delta^s a= \delta^s\delta^k a$ for any
$k,s= 1,2,...,m$, which completes the proof of Proposition 2.1.\end{proof}

Further it is assumed that the commuting system of differential
operators $\partial^1,. . .,\partial^m$ is linear independent over
$F$. Note that for any $g\in GL^{\partial}(m,F)$ and $\delta=
g^{-1}\partial$ the constant field of $(F;\delta)$ is the same $C$ and equality
\begin{eqnarray}\begin{array}{c} GL^{\delta}(m,F)=
g^{-1}GL^{\partial}(m,F) \end{array}\end{eqnarray} holds true.

Further each element $(h,g)\in GL(n,F)\times GL^{\partial}(m,F)$ is assumed to be presented as a block-diagonal matrix with the blocks "$h$"\ and\ "$g$"\ on the main diagonal.

We consider the following groupoid  $G$ (generated by $GL^{\partial}(m,F)$ and $GL(n,F)$) as a fiber bundled space with base \[\mathbf{B}=\{\delta: \mbox{there exists}\ g\in GL^{\partial}(m,F)\ \mbox{such that}\ \delta=g^{-1}\partial\},\] and fiber $GL(n,F)\times GL^{\delta}(m,F)$ over $\delta\in \mathbf{B}$ and maps $s: G\rightarrow \mathbf{B}$, $t: G\rightarrow \mathbf{B}$, where by definition $s(\delta,(h,g))=\delta$ and $t(\delta,(h,g))=g^{-1}\delta$. Whenever $s(\delta_1,(h_1,g_1))= t(\delta,(h,g))$ the product $(\delta,(h,g))(\delta_1,(h_1,g_1))=((g_1)^{-1}\delta,(hh_1, gg_1))$
is defined. All needed axioms of groupoid \cite{Br, We} are checked easily.

We consider also fiber bundled space $DE$ with the same base $\mathbf{B}$ and fiber $F^l=F_{\delta}^l$ over $\delta\in \mathbf{B}$, where $F_{\delta}$ means the field $F$ considered as $\delta$-differential field. The algebra of polynomials $C\{\mathbf{x};\partial\}$, where $\mathbf{x}=(x_1, x_2, ..., x_l)$, can be extended to the $C\{DE\}$-the algebra of differential polynomials on $DE$ in the following natural way: $f^{\partial}\{\mathbf{x}\}\in C\{\mathbf{x};\partial\}$ generates polynomial map $(f^{\delta}\{\mathbf{x}\})_{\delta\in \mathbf{B}}\in C\{DE\}$, where $f^{\delta}\{\mathbf{x}\} $ stands for the polynomial obtained by substitution $\delta$ for $\partial$ into $f^{\partial}\{\mathbf{x}\}$. In other words each element of $C\{DE\}$ when considered on $(\delta,F_{\delta}^l)$ is a $\delta$-differential polynomial and its coefficients do not depend on $\delta$. By $C\langle DE\rangle$ we denote the corresponding field of rational functions. Let us use $I_n$ to denote the $n$-order identity matrix and we use $e$ for $I_m$ also. We consider  $W=\bigcup_{\delta\in \mathbf{B}}(\delta,W_{\delta})\subset  DE$ a sub-fiber bundled space and $C\{W\}$, $C\langle W\rangle$ in the natural way.

Let $G\diamond W$ stand for the fiber bundled space with the same base $\mathbf{B}$ and fiber $GL(n,F)\times GL^{\delta}(m,F)\times W_{\delta}$ over $\delta$.

By algebraic ("anti") representation of $G$ on $W$ we mean a map $T:G\diamond W \rightarrow W $, where $T:(\delta, (h,g,v))\mapsto (g^{-1}\delta, \tau^{\delta}\langle h,g;v\rangle)$, such that
\begin{equation}\begin{array}{c} \tau^{g^{-1}\delta}\langle h_1,g_1; \tau^{\delta}\langle h,g;v\rangle\rangle= \tau^{\delta}\langle hh_1,gg_1;v\rangle\end{array}\end{equation}
whenever $g\in GL^{\delta}(m,F)$, $v\in W_{\delta}$, $h,h_1\in GL(n,F)$ and $g_1\in GL^{g^{-1}\delta}(m,F)$. It is assumed that $\tau^{\delta}\langle I_n,I_m;v\rangle=v$ for any $(\delta,v)$.

We say that \[W_0=\bigcup_{\delta\in \mathbf{B}}(\delta,W_{0\delta})\subset W\ \mbox{is invariant}\ (G-\mbox{invariant) if}\ \tau^{\delta}\langle h,g;v\rangle\in W_{0g^{-1}\delta}\] whenever $(h,g)\in GL(n,F)\times GL^{\delta}(m,F)$ and $\mathbf{v}\in W_{0\delta}
$.

 Further it is assumed that each component of  $(\tau^{\delta})_{\delta\in \mathbf{B}}$ is from $C\langle W\rangle$ that is a differential rational function in components of $h,g,v$ over $C$ .

{\it {\bf Assumption.} There exists a nonempty $G$-invariant subset $W_0$ of $W$ and a map $P=(P^{\delta})_{\delta\in \mathbf{B}}: W_0\rightarrow G$, where  $P: (\delta, v)\mapsto (\delta, P^{\delta}\langle\mathbf{v}\rangle)$, such that
 \begin{equation}P^{g^{-1}\delta}\langle\tau^{\delta}\langle h,g;\mathbf{v}\rangle\rangle=(h^{-1},g^{-1})P^{\delta}\langle\mathbf{v}\rangle\end{equation}
 whenever $\delta\in \mathbf{B}$, $\mathbf{v}\in W_{\delta}$ and $(h,g)\in GL(n,F)\times GL^{\delta}(m,F)$. It is also assumed that all components of $(P^{\delta}\langle\mathbf{v}\rangle)_{\delta\in \mathbf{B}}$ are differential rational functions in $\mathbf{v}$ over $C$.}

 In terms of components of $P^{\delta}\langle\mathbf{v}\rangle=(P_0^{\delta}\langle\mathbf{v}\rangle,P_1^{\delta}\langle\mathbf{v}\rangle)$ equality (2.4) can be written in the form:

 \[P_0^{g^{-1}\delta}\langle\tau^{\delta}\langle h,g;\mathbf{v}\rangle\rangle=h^{-1}P_0^{\delta}\langle\mathbf{v}\rangle,\ P_1^{g^{-1}\delta}\langle\tau^{\delta}\langle h,g;\mathbf{v}\rangle\rangle=g^{-1}P_1^{\delta}\langle\mathbf{v}\rangle.\]

Note that $G$-stabilizer of any $(\delta,\mathbf{v})\in W_0$ is trivial. The next result is a classification theorem for the elements of $W_0$ with respect to $G$.

\begin{theorem} Pairs $(\partial,\mathbf{u}), (\delta,\mathbf{v})\in W_0$ are $G$-equivalent, that is $\mathbf{v}=\tau^{\partial}\langle h,g;\mathbf{u}\rangle$ and $\delta=g^{-1}\partial$ for some $(h,g)\in GL(n,F)\times GL^{\partial}(m,F)$, if and only if \[P_1^{\partial}\langle\mathbf{u}\rangle^{-1}\partial=P_1^{\delta}\langle\mathbf{v}\rangle^{-1}\delta\ \ \mbox{and}\ \ \tau^{\partial}\langle P^{\partial}\langle\mathbf{u}\rangle ;\mathbf{u}\rangle=\tau^{\delta}\langle P^{\delta}(\mathbf{v}\rangle;\mathbf{v}\rangle.\]\end{theorem}

\begin{proof} If $\mathbf{v}=\tau^{\partial}\langle h,g;\mathbf{u}\rangle$ and $\delta=g^{-1}\partial$ then \[P_1^{\delta}\langle\mathbf{v}\rangle^{-1}\delta=(P_1^{g^{-1}\partial}\langle\tau^{\partial}\langle h,g;\mathbf{u}\rangle\rangle)^{-1}g^{-1}\partial=(g^{-1}P_1^{\partial}\langle \mathbf{u}\rangle)^{-1}g^{-1}\partial=P_1^{\partial}\langle \mathbf{u}\rangle^{-1}\partial\] and
\[\tau^{\delta}\langle P^{\delta}\langle\mathbf{v}\rangle;\mathbf{v}\rangle=\tau^{\delta}\langle P^{g^{-1}\partial}\langle\tau^{\partial}\langle h,g;\mathbf{u}\rangle\rangle;\mathbf{v}\rangle=
\tau^{\delta}\langle (h^{-1},g^{-1})P^{\partial}\langle\mathbf{u}\rangle;\mathbf{v}\rangle=\]
\[\tau^{g^{-1}\partial}\langle (h^{-1},g^{-1})P^{\partial}\langle\mathbf{u}\rangle;\tau^{\partial}\langle h,g;\mathbf{u}\rangle\rangle=\tau^{\partial}\langle P^{\partial}\langle\mathbf{u}\rangle;\mathbf{u}\rangle\]

Visa versa, if $P_1^{\partial}\langle\mathbf{u}\rangle^{-1}\partial=P_1^{\delta}\langle\mathbf{v}\rangle^{-1}\delta$ and $\tau^{\partial}\langle P^{\partial}\langle\mathbf{u}\rangle ;\mathbf{u}\rangle=\tau^{\delta}\langle P^{\delta}\langle\mathbf{v}\rangle;\mathbf{v}\rangle$  then for $h=P_0^{\partial}\langle\mathbf{u}\rangle P_0^{\delta}\langle\mathbf{v}\rangle^{-1}, g=P_1^{\partial}\langle\mathbf{u}\rangle P_1^{\delta}\langle\mathbf{v}\rangle^{-1}$ due to (2.2), (2.3) and the condition of the theorem  one has
\[\tau^{\partial}\langle h,g;\mathbf{u}\rangle=\tau^{\partial}\langle P_0^{\partial}\langle\mathbf{u}\rangle P_0^{\delta}\langle\mathbf{v}\rangle^{-1},
P_1^{\partial}\langle\mathbf{u}\rangle P_1^{\delta}\langle\mathbf{v}\rangle^{-1};\mathbf{u}\rangle=\]
\[\tau^{P_1^{\partial}\langle\mathbf{u}\rangle^{-1}\partial}\langle P_0^{\delta}\langle\mathbf{v}\rangle^{-1},P_1^{\delta}\langle\mathbf{v}\rangle^{-1};\tau^{\partial}
\langle P_0^{\partial}\langle\mathbf{u}\rangle,P_1^{\partial}\langle\mathbf{u}\rangle;\mathbf{u}\rangle\rangle=\]
\[\tau^{P_1^{\delta}\langle\mathbf{v}\rangle^{-1}\delta}\langle P_0^{\delta}\langle\mathbf{v}\rangle^{-1},P_1^{\delta}
\langle\mathbf{v}\rangle^{-1};\tau^{\delta}\langle P_0^{\delta}\langle\mathbf{v}\rangle,P_1^{\delta}\langle\mathbf{v}\rangle;\mathbf{v}\rangle\rangle=\tau^{\delta}\langle I_n, I_m;\mathbf{v}\rangle=\mathbf{v}.\] \end{proof}

This theorem means that for the canonical representative of $G$-equivalent to $(\partial, \mathbf{u})$ elements of $W_0$ one can take
\[(P_1^{\partial}\langle\mathbf{u}\rangle^{-1}\partial, \tau^{\partial}\langle P^{\partial}\langle\mathbf{u}\rangle ;\mathbf{u}\rangle).\]

The next result is a criterion for reducibility(equivalence) to an element with constant components.

\begin{theorem} An element $(\partial,\mathbf{u})\in W_{0}$ is $G$-equivalent to an element $(\delta,\mathbf{v})\in W_{0}$, where all components of $\mathbf{v}$ are constants, that is $\mathbf{v}\in C^l$, if and only if  all components of \[\tau^{\partial}\langle P^{\partial}\langle\mathbf{u}\rangle;\mathbf{u}\rangle\] are constants.\end{theorem}

\begin{proof} The "If" part. If $\mathbf{v}=\tau^{\partial}\langle h,g;\mathbf{u}\rangle\in C^l$ for some $(h,g)\in GL(n,F)\times GL^{\partial}(m,F) $ then all entries of matrix $P^{g^{-1}\partial}\langle\tau^{\partial}\langle h,g;\mathbf{u}\rangle\rangle=(h^{-1},g^{-1})P^{\partial}\langle\mathbf{u}\rangle$ are in $C$ as well. That is at any $i=1,2,...,m$ one has \[0=\partial^i((h^{-1},g^{-1})P^{\partial}\langle\mathbf{u}\rangle)=-(h^{-1}\partial^i hh^{-1},g^{-1}\partial^i g g^{-1})P^{\partial}\langle\mathbf{u}\rangle+(h^{-1},g^{-1})\partial^i P^{\partial}\langle\mathbf{u}\rangle\] which implies equality \[(\partial^i hh^{-1},\partial^i g g^{-1})=\partial^i P^{\partial}\langle\mathbf{u}\rangle P^{\partial}\langle\mathbf{u}\rangle^{-1}.\]
Therefore one has $(h,g)=P^{\partial}\langle\mathbf{u}\rangle(h_0,g_0)$ for some $h_0\in GL(n,C)$ and $g_0\in GL(m,C)\subset GL^{\partial}(m,F)$ as far as the solution of the system
\[\partial^i XX^{-1}=\partial^i P^{\partial}\langle\mathbf{u}\rangle P^{\partial}\langle\mathbf{u}\rangle^{-1}, \ i=1,2,...,m\] is unique up to such product. The equality \[\mathbf{v}=\tau^{\partial}\langle h,g;\mathbf{u}\rangle=\tau^{g_0^{-1}\partial}\langle h_0,g_0;\tau^{\partial}\langle P^{\partial}\langle\mathbf{u}\rangle;\mathbf{u}\rangle\rangle\] indicates that $\tau^{\partial}\langle P^{\partial}\langle\mathbf{u}\rangle;\mathbf{u}\rangle\rangle=\tau^{\partial}\langle h_0^{-1},g_0^{-1};\mathbf{v}\rangle\in C^l$ as well.

The "only if" part is evident. \end{proof}

 Further in this section $(F,\partial)$ is considered to be a differential-algebraic closed field of characteristic zero, $W_{\delta}$ is assumed to be a closed, irreducible (in the differential Zariski topology) subset of $F_{\delta}^l$ and $W_{0\delta}$ is dense in $W_{\delta}$.

The following set $C\langle W\rangle^G=\{f=(f^{\delta}\langle \mathbf{x}\rangle)_{\delta}\in C\langle W\rangle:$  \[f^{g^{-1}\delta}\langle \tau^{\delta}\langle h, g;\mathbf{v}\rangle\rangle=f^{\delta}\langle \mathbf{v}\rangle\ \mbox{whenever}\ (h,g,v)\in GL(n,F)\times GL^{\delta}(m,F)\times W_{\delta}, \delta\in \mathbf{B}\}\] is said to be the field of $G$-invariant rational functions on $W$. From the very definition it is clear that
any $f=(f^{\delta}\langle \mathbf{x}\rangle)_{\delta}\in C\langle W\rangle^G$ is fully defined by its fiber $f^{\partial}\langle \mathbf{x}\rangle$. Therefore for this set one can use more informative and transparent notation  $C\langle W_e;\partial\rangle^{GL(n,F)\times GL^{\delta}(m,F)}$ or $C\langle W_e;\partial\rangle^G$- the field of $G$-invariant rational functions on $W_e$.

\begin{theorem} The field of $G$-invariant rational functions $C\langle W_e;\partial\rangle^G$ is invariant with respect to the commuting system of differential operators $\delta^1, \delta^2,
..., \delta^m$, where $\delta=P_1^{\partial}\langle\mathbf{x}\rangle^{-1}\partial$, and as such differential field it is generated over $C$ by the system of components of $\tau^{\partial}\langle P^{\partial}\langle\mathbf{x}\rangle ;\mathbf{x}\rangle$ that is  \[C\langle W_e;\partial\rangle^G=C\langle\tau^{\partial}\langle P^{\partial}\langle\mathbf{x}\rangle ;\mathbf{x}\rangle; \delta\rangle.\]\end{theorem}

\begin{proof} It is evident that all components of  $\tau^{\partial}\langle P^{\partial}\langle\mathbf{x}\rangle ;\mathbf{x}\rangle$ are in $ C\langle W_e;\partial\rangle^G$. If $f^{\partial}\langle\mathbf{x}\rangle=f^{g^{-1}\partial}\langle\tau^{\partial}\langle h,g ;\mathbf{x}\rangle\rangle$ for all $(h,g)\in GL(n,F)\times GL^{\partial}(m,F)$ then, in particular, \[f^{\partial}\langle\mathbf{x}\rangle=f^{P_1^{\partial}\langle\mathbf{u}\rangle^{-1}\partial}\langle \tau^{\partial}\langle P_0^{\partial}\langle\mathbf{u}\rangle, P_1^{\partial}\langle\mathbf{u}\rangle;\mathbf{x}\rangle\rangle\]  whenever $\mathbf{u}\in W_{0e}$. It implies, as far as $W_{0e}$ is dense in $W_e$, that for the variable vector $\mathbf{y}\in W_0$ the equality \[f^{\partial}\langle\mathbf{x}\rangle=f^{P_1^{\partial}\langle\mathbf{y}\rangle^{-1}\partial}\langle \tau^{\partial}\langle P_0^{\partial}\langle\mathbf{y}\rangle, P_1^{\partial}\langle\mathbf{y}\rangle;\mathbf{x}\rangle\rangle\]
 holds true. In $\mathbf{y}=\mathbf{x}$ case one gets that $f^{\partial}\langle\mathbf{x}\rangle=f^{P_1^{\partial}\langle\mathbf{x}\rangle^{-1}\partial}\langle \tau^{\partial}\langle P_0^{\partial}\langle\mathbf{x}\rangle, P_1^{\partial}\langle\mathbf{x}\rangle;\mathbf{x}\rangle\rangle$. \end{proof}

\begin{corollary} The field $C\langle W_e;\partial\rangle$ is generated over $C\langle W_e;\partial\rangle^G$ as a $\delta=P_1^{\partial}\langle\mathbf{x}\rangle^{-1}\partial$-differential field by the system of components of $P^{\partial}\langle\mathbf{x}\rangle$ that is \[C\langle W_e;\partial\rangle=C\langle W_e;\partial\rangle^G \langle P^{\partial}\langle\mathbf{x}\rangle; \delta \rangle.\]\end{corollary}

\begin{proof} Due to Theorem 2.4 one has equality \[C\langle W_e;\partial\rangle^G\langle P^{\partial}\langle\mathbf{x}\rangle;P_1^{\partial}\langle\mathbf{x}\rangle^{-1}\partial\rangle= C\langle \tau^{\partial}\langle P^{\partial}\langle\mathbf{x}\rangle ;\mathbf{x}\rangle, P^{\partial}\langle\mathbf{x}\rangle; P_1^{\partial}\langle\mathbf{x}\rangle^{-1}\partial\rangle.\] Therefore to prove the corollary it is enough to show that $\partial$, as well as $\mathbf{x}$, can be expressed by $\delta=P_1^{\partial}\langle\mathbf{x}\rangle^{-1}\partial$ and the system of components of $P^{\partial}\langle\mathbf{x}\rangle$ and $\tau^{\partial}\langle P^{\partial}\langle\mathbf{x}\rangle ;\mathbf{x}\rangle$. Indeed $\partial=P_1^{\partial}\langle\mathbf{x}\rangle\delta$ and \[\tau^{P_1^{\partial}\langle\mathbf{x}\rangle^{-1}\partial}
\langle P^{\partial}\langle\mathbf{x}\rangle^{-1};\tau^{\partial}\langle P^{\partial}\langle\mathbf{x}\rangle;\mathbf{x}\rangle \rangle=\mathbf{x}.\]\end{proof}

\begin{theorem} The equality $\delta-\mbox{tr.deg.}C\langle W_e;\partial\rangle^G/C=\dim (W_e)-(n^2+m)$ holds true.\end{theorem}

 \begin{proof} We will show that the system \[P_0^{\partial}\langle\mathbf{x}\rangle, P_{1,1}^{\partial}\langle\mathbf{x}\rangle, P_{1,2}^{\partial}\langle\mathbf{x}\rangle,...,P_{1,m}^{\partial}\langle\mathbf{x}\rangle\] is $\delta$-algebraic independent over $C\langle W_e;\partial\rangle^G$, where
 $P_0^{\partial}\langle\mathbf{x}\rangle$ stands for all its components and
 $P_{1,1}^{\partial}\langle\mathbf{x}\rangle, P_{1,2}^{\partial}\langle\mathbf{x}\rangle,...,P_{1,m}^{\partial}\langle\mathbf{x}\rangle$ stands for the first row of the matrix $P_1^{\partial}\langle\mathbf{x}\rangle$. Note that the system can be presented in the form $e_1P^{\partial}\langle\mathbf{x}\rangle$, where
 $e_1$ stands for  $(I_n,1,0,0,...,0)$.

If $p^{\delta}[Z_0,z_1,...,z_m]$ is any $\delta$-differential polynomial over $C\langle W_e;\partial\rangle^G$ for which $p^{\delta}[ e_1P^{\partial}\langle\mathbf{x}\rangle]=0$
then it should remain be true with respect to the changes $\partial\mapsto g^{-1}\partial$ and $\mathbf{x}\mapsto \tau^{\partial}\langle h,g;\mathbf{x}\rangle$, where $h\in GL(n,F)$ and $g\in GL^{\partial}(m,F)$. As far as $\delta$ and the coefficients of $p^{\delta}[Z_0,z_1,...,z_m]$ are invariant with respect to such changes one has
\[ 0=p^{\delta}[e_1P^{\partial}\langle\mathbf{x}\rangle]=p^{\delta}[e_1P^{g^{-1}\partial}\langle\tau^{\partial}\langle h,g;\mathbf{x}\rangle\rangle]= p^{\delta}[e_1(h^{-1},g^{-1})P^{\partial}\langle\mathbf{x}\rangle]\] that is $p_v^{\delta_v}[e_1(h^{-1},g^{-1})P^{\partial}\langle\mathbf{v}\rangle]=0$ for any $v\in W_{0e}$, where $p_v^{\delta_v}[Z_0,z_1,...,z_m]$ stands for the polynomial obtained from $p^{\delta}[Z_0,z_1,...,z_m]$ by substitution $\mathbf{v}$ for $\mathbf{x}$. But for any $h_0\in GL(n,F)$ and $\overline{g}\in GL^{\delta_v}(m,F)$ the system $h_0=h^{-1}P_0^{\partial}\langle\mathbf{v}\rangle$,
$\overline{g}=g^{-1}P_1^{\partial}\langle\mathbf{v}\rangle$ has solution $h=P_0^{\partial}\langle\mathbf{v}\rangle h_0^{-1}\in GL(n,F)$, $g=P_1^{\partial}\langle\mathbf{v}\rangle \overline{g}^{-1}\in GL^{\partial}(m,F)$. Therefore $p_v^{\delta_v}[e_1(h_0,\overline{g})]=0$ for any $h_0\in GL(n,F)$ and $\overline{g}\in GL^{\delta_v}(m,F)$. In particular it implies that for any $\overline{g}=\delta_v\mathbf{u}\in  GL^{\delta_v}(m,F)$ the equality
\begin{equation} p_v^{\delta_v}[h_0,\delta_v^1u_1, \delta_v^1u_2,...,\delta_v^1u_m]=0\end{equation} holds true, where $\overline{g}=\delta_v\mathbf{u}=\delta_v(u_1,u_2,...,u_m)$ stands for the square matrix $(\delta_v^iu_j)_{i,j=1,2,...,m}$ . Note that if $p_v^{\delta_v}[Z_0,z_1,...,z_m]$ is nonzero then the same is true for $p_v^{\delta_v}[Z_0,\delta_v^1z_1,...,\delta_v^1z_m]$ . Equality (2.5) shows that the $\delta_v$-differential polynomial \[\det Z_0\det(\delta_v(z_1,z_2,...,z_m))p_v^{\delta_v}[Z_0,\delta_v^1z_1, \delta_v^1z_2,...,\delta_v^1z_m]\] vanishes at any $(Z_0,z_1,z_2,...,z_m)\in F^{n^2+m}$. As far as $\det Z_0\det(\delta_v(z_1,z_2,...,z_m))$ is a nonzero polynomial one can conclude that $p_v^{\delta_v}[Z_0,\delta_v^1z_1, \delta_v^1z_2,...,\delta_v^1z_m]$ has to be zero polynomial at any $\mathbf{v}\in W_{0e}$. Therefore $p^{\delta}[Z_0,\delta^1z_1, \delta^1z_2,...,\delta^1z_m]=0$ that is $p^{\delta}[Z_0,z_1,...,z_m]$ itself is zero polynomial. Due to the equality \[C\langle W_e;\partial\rangle^G \langle P^{\partial}\langle\mathbf{x}\rangle; \delta \rangle=C\langle W_e;\partial\rangle^G \langle P^{\partial}\langle\mathbf{x}\rangle^{-1}; \delta \rangle \ \mbox{and}\ P_1^{\partial}\langle\mathbf{x}\rangle^{-1}\in GL^{\delta}\langle m, C\langle\mathbf{x};\partial\rangle\rangle\] the $\delta$-algebraic transcendence degree of the system of components of $P^{\partial}\langle\mathbf{x}\rangle$ over
$C\langle W_e;\partial\rangle^G$  equals exactly to $n^2+m$. Now taking into account the equalities \[\dim W_e=\partial-\mbox{tr.deg.}C\langle W_e;\partial\rangle/C=\delta-\mbox{tr.deg.}C\langle W_e;\partial\rangle/C\] one can conclude that the equality $\delta-\mbox{tr.deg.}C\langle W_e;\partial\rangle^G/C=\dim (W_e)-(n^2+m)$ is true. \end{proof}

\section{Classification and invariants of the second order non-parabolic LPDE}

In this section we consider applications of the previous section results to the corresponding problems for the second order LPDE in two variables, which corresponds to $n=1, m=2, l=6$ case of that section. We get the needed applications via the construction the corresponding map $P^{\partial}$ satisfying conditions of the Assumption. In general neither $W_{0e}$ nor $P^{\partial}$ is unique.Therefore our construction of $P^{\partial}$ should be considered as one of the possible constructions. We refrain from reformulation  all results of Section 2 for each particular case as far as though the reformulation is not difficult but space consuming process. Further we use the following functions

 $\begin{array}{l}
  a^{\partial}\langle  V\rangle=  AC_x - A_xC + BC_y - B_yC - bB + 2aC ,\\
 b^{\partial}\langle  V\rangle= A_xB - AB_x + A_yC - AC_y - aB+ 2bA ,\\
  c^{\partial}\langle  V\rangle= ((\frac{a^{\partial}\langle  V\rangle}{D\langle V\rangle})_y - (\frac{b^{\partial}\langle  V\rangle}{D\langle  V\rangle})_x)^2D\langle V\rangle,\\
   c_1^{\partial}\langle  V\rangle=-\frac{(c^{\partial}\langle  V\rangle)_x}{2c^{\partial}\langle  V\rangle},\\
  c_2^{\partial}\langle  V\rangle=-\frac{(c^{\partial}\langle  V\rangle)_y}{2c^{\partial}\langle  V\rangle},\\
 \gamma^{\partial}\langle  V\rangle=A((c_1^{\partial}\langle  V\rangle)_x + c_1^{\partial}\langle  V\rangle^2)+B((c_1^{\partial}\langle  V\rangle)_y + c_1^{\partial}\langle  V\rangle c_2^{\partial}\langle  V\rangle) + \\ C((c_2^{\partial}\langle  V\rangle)_y +
  c_2^{\partial}\langle  V\rangle^2) + ac_1^{\partial}\langle  V\rangle + bc_2^{\partial}\langle  V\rangle+c.\\
 \gamma_0^{\partial}\langle  V\rangle=A((\frac{a^{\partial}\langle  V\rangle}{D\langle V\rangle})_x + \frac{a^{\partial}\langle  V\rangle^2}{D\langle V\rangle^2})+B((\frac{a^{\partial}\langle V\rangle}{D\langle V\rangle})_y + \frac{a^{\partial}\langle  V\rangle}{D\langle V\rangle} \frac{b^{\partial}\langle  V\rangle}{D\langle  V\rangle})+ \\ C((\frac{b^{\partial}\langle  V\rangle}{D\langle  V\rangle})_y + \frac{b^{\partial}\langle  V\rangle^2}{D\langle  V\rangle^2}) +
 a\frac{a^{\partial}\langle  V\rangle}{D\langle V\rangle} + b\frac{b^{\partial}\langle  V\rangle}{D\langle  V\rangle} + c,\ \mbox{where}\ D\langle V\rangle=B^2-4AB, \\
\end{array}$

to investigate classification and invariants problems for the second order LPDE. These functions are introduced in \cite{B1}.

The following theorem from \cite{B1} reduces system of equalities (1.3) to a more simple equivalence under the condition $ D\langle V\rangle \neq0 $.

\begin{theorem} If $ D\langle V\rangle\neq 0$ and $c^{\partial}\langle  V\rangle \neq0 $ ($ c^{\partial}\langle  V\rangle =0$) then system of equalities (1.3) is equivalent to the following system of equalities.
\begin{eqnarray}
\left\{ \begin{array}{l} \left( \begin{array}{cc}
A_{1} & B_{1}/2  \\
B_{1}/2 & C_{1}
\end{array} \right) =h
g^t \left(\begin{array}{cc}
A & B/2  \\
B/2 & C
\end{array} \right) g, \\
\left( \begin{array}{l}\frac{a^{\delta}\langle  V_1\rangle}{D\langle V_1\rangle}\\ \frac{b^{\delta}\langle  V_1\rangle}{D\langle V_1\rangle}\\ \end{array}\right)=
g^{-1}\left( \begin{array}{l}\frac{a^{\partial}\langle  V\rangle}{D\langle V\rangle}- h^{-1}\partial^1h\\ \frac{b^{\partial}\langle  V\rangle}{D\langle V\rangle} - h^{-1}\partial^2h\\ \end{array}\right),\\
\gamma^{\delta}\langle V_1\rangle=h\gamma^{\partial}\langle V\rangle\ \
(\mbox{respectively,}\ \gamma_0^{\delta}\langle V_1\rangle=h\gamma_0^{\partial}\langle V\rangle.)\\
\end{array} \right.
\end{eqnarray}
, where
 $g^t$ stands for the transpose of $g$.\end{theorem}

 One of the interesting questions in theory of LPDE is the equivalence problem of LPDE to LPDE with constant coefficients. For a solution of this question for any order ordinary LDE one can see \cite{B4, B5, B6}. Here we would like to show that the above result can be used to treat this problem even without construction of the corresponding map $P^{\partial}\langle  \mathbf{x}\rangle$ (see Theorem 2.3).
 Note that for non-parabolic equation (1.1) the condition $V\in C^6$ is equivalent to \[(A,B,C,\frac{a^{\partial}\langle  V\rangle}{D\langle V\rangle},\frac{b^{\partial}\langle  V\rangle}{D\langle V\rangle},\gamma_0^{\partial}\langle V\rangle)\in C^6,\]  which can be seen directly from
the definition of the functions $a^{\partial}\langle  V\rangle,b^{\partial}\langle  V\rangle,\gamma_0^{\partial}\langle V\rangle$.  It implies that for (1.1) to be $G$-equivalent to LPDE with constant coefficients at least $c^{\partial}\langle  V\rangle$ should be zero that is the equality \[\partial^2\frac{a^{\partial}\langle  V\rangle}{D\langle V\rangle}=\partial^1\frac{b^{\partial}\langle  V\rangle}{D\langle V\rangle}\] is necessary.

\begin{corollary} Non-parabolic equation (1.1), for which $\partial^2\frac{a^{\partial}\langle  V\rangle}{D\langle V\rangle}=\partial^1\frac{b^{\partial}\langle  V\rangle}{D\langle V\rangle}$ and $\gamma_0^{\partial}\langle V\rangle\neq 0$, is $G$-equivalent to a LPDE with constant coefficients if and only if the system
 \[ \begin{array}{l}(\partial^igg^{-1})^tM^{\partial}\langle V\rangle + M^{\partial}\langle V\rangle\partial^igg^{-1} + \partial^iM^{\partial}\langle V\rangle=0, i=1,2 \\
\partial^igg^{-1}m^{\partial}\langle V\rangle = \partial^im^{\partial}\langle V\rangle, i=1,2 \end{array}\] has a solution in $GL^{\partial}(2,F)$,
where
$M^{\partial}\langle V\rangle= \left( \begin{array}{cc}
A/\gamma_0^{\partial}\langle V\rangle & B/(2\gamma_0^{\partial}\langle V\rangle)  \\
B/(2\gamma_0^{\partial}\langle V\rangle) & C/\gamma_0^{\partial}\langle V\rangle
\end{array} \right),\\
m^{\partial}\langle V\rangle= \left( \begin{array}{c} \frac{a^{\partial}\langle  V\rangle}{D\langle V\rangle}+ \gamma_0^{\partial}\langle V\rangle^{-1}\partial^1\gamma_0^{\partial}\langle V\rangle \\
\frac{b^{\partial}\langle  V\rangle}{D\langle V\rangle}+ \gamma_0^{\partial}\langle V\rangle^{-1}\partial^2\gamma_0^{\partial}\langle V\rangle\end{array} \right)$.
 \end{corollary}

\begin{proof} The relation $V_1=\tau^{\partial}\langle h,g; V\rangle \in C^6$ is equivalent to \[ (A_1,B_1,C_1,\frac{a^{\delta}\langle  V_1\rangle}{D\langle V_1\rangle},\frac{b^{\delta}\langle  V_1\rangle}{D\langle V_1\rangle},\gamma_0^{\delta}\langle V_1\rangle)\in C^6.\] Therefore due to Theorem 3.1 it is equivalent to
$h^{-1}\partial h=-\gamma_0^{\partial}\langle V\rangle^{-1}\partial\gamma_0^{\partial}\langle V\rangle$ and
\[ \left\{ \begin{array}{l}\partial^i(g^tM^{\partial}\langle V\rangle g)=0, i=1,2 \\
\partial^i(g^{-1}m^{\partial}\langle V\rangle) = 0, i=1,2 \end{array} \right.\] which can easily be presented in the following form \[ \left\{ \begin{array}{l}(\partial^igg^{-1})^tM^{\partial}\langle V\rangle + M^{\partial}\langle V\rangle\partial^igg^{-1} + \partial^iM^{\partial}\langle V\rangle=0, i=1,2 \\
\partial^igg^{-1}m^{\partial}\langle V\rangle = \partial^im^{\partial}\langle V\rangle , i=1,2.  \end{array}\right..\] \end{proof}

It should be noted that matrices $\partial^1gg^{-1}, \partial^2gg^{-1}$ are defined uniquely from the linear system  \begin{eqnarray} (\partial^igg^{-1})^tM^{\partial}\langle V\rangle + M^{\partial}\langle V\rangle\partial^igg^{-1} + \partial^iM^{\partial}\langle V\rangle=0, i=1,2 \end{eqnarray} due to $D\langle V\rangle\neq 0$. In terms of the ordinary notation $ \left( \begin{array}{cc}
\xi_x & \eta_x \\
\xi_y & \eta_y
\end{array} \right)$ for $g$ it means that expressions \begin{eqnarray}\begin{array}{c} \frac{1}{\Delta}|\begin{array}{cc}
\xi_{xx} & \eta_{xx} \\
\xi_y & \eta_y
\end{array} |,\ \frac{1}{\Delta}|\begin{array}{cc}\xi_{x} & \eta_{x} \\
\xi_{xx} & \eta_{xx}\end{array} |,\  \frac{1}{\Delta}|\begin{array}{cc}
\xi_{xy} & \eta_{xy} \\ \xi_y & \eta_y\end{array} |, \\ \frac{1}{\Delta}|\begin{array}{cc}
\xi_{x} & \eta_{x} \\
\xi_{xy} & \eta_{xy}\end{array} |,\ \frac{1}{\Delta}|\begin{array}{cc}
\xi_{yy} & \eta_{yy} \\
\xi_y & \eta_y
\end{array} |,\ \frac{1}{\Delta}|\begin{array}{cc}
\xi_{x} & \eta_{x} \\
\xi_{yy} & \eta_{yy}
\end{array} |,\end{array}\end{eqnarray} which represent all entries of the matrices $\partial^1gg^{-1}, \partial^2gg^{-1}$, are defined uniquely by the entries of $M^{\partial}\langle V\rangle$.

Therefore system (3.2) can be written in the form \begin{eqnarray}\{\begin{array}{c}\partial^1gg^{-1}=N_1^{\partial}\langle V\rangle\\ \partial^2gg^{-1}=N_2^{\partial}\langle V\rangle\end{array}.\end{eqnarray}
In particular it implies that equalities $\partial^igg^{-1}m^{\partial}\langle V\rangle = \partial^im^{\partial}\langle V\rangle, i=1,2$ are nothing than some relations between entries of $M^{\partial}\langle V\rangle$ and
$m^{\partial}\langle V\rangle$ needed for the reducibility of (1.1) to LPDE with constant coefficients. So the problem is reduced to the solvability of (3.4) in  $GL^{\partial}(2,F)$.

\begin{remark} One can ask about the admissible values of six expressions given in (3.3) . That is if one attaches some values to these six expressions how to know whether the obtained system of first order nonlinear differential equations has solution in $GL^{\partial}(2,F)$ or not? In other words one can ask about the integrability conditions of the obtained system of six equations. It can be answered in the following way. Construct matrices $N_1^{\partial}\langle V\rangle, N_2^{\partial}\langle V\rangle$ using the attached values and equalities (3.4). The integrability condition for the system is given by the next Proposition.\end{remark}

In the next Proposition it is assumed that the matrix equation $\partial^1gg^{-1}=N$ has solution in $GL(2,F)$ whenever $N$ is a second order matrix over $F$ and if all entries of $N$ are constants with respect to $\partial^1$, that is $\partial^1N=0,$ then the equation  $\partial^2gg^{-1}=N$ also has solution with constant entries with respect to $\partial^1$.

\begin{proposition} System (3.4) has solution in $GL^{\partial}(2,F)$ if and only if the following equalities \[\partial^1N_2^{\partial}\langle V\rangle + N_2^{\partial}\langle V\rangle N_1^{\partial}\langle V\rangle =\partial^2N_1^{\partial}\langle V\rangle + N_1^{\partial}\langle V\rangle N_2^{\partial}\langle V\rangle,\ \ N_2^{\partial}\langle V\rangle^{(1)}=N_1^{\partial}\langle V\rangle^{(2)}\] hold true, where $N_2^{\partial}\langle V\rangle^{(1)}$ stands for the first row of the matrix $N_2^{\partial}\langle V\rangle$.\end{proposition}

\begin{proof} If (3.4) is true for some $g\in GL^{\partial}(2,F)$ then \[\partial^2N_1^{\partial}\langle V\rangle=\partial^2(\partial^1gg^{-1})=\partial^2\partial^1gg^{-1}- \partial^1gg^{-1}\partial^2gg^{-1}=\partial^2\partial^1gg^{-1}- N_1^{\partial}\langle V\rangle N_2^{\partial}\langle V\rangle\] and similarly \[\partial^1N_2^{\partial}\langle V\rangle=\partial^2\partial^1gg^{-1}- N_2^{\partial}\langle V\rangle N_1^{\partial}\langle V\rangle\] which imply that \[\partial^1N_2^{\partial}\langle V\rangle + N_2^{\partial}\langle V\rangle N_1^{\partial}\langle V\rangle =\partial^2N_1^{\partial}\langle V\rangle + N_1^{\partial}\langle V\rangle N_2^{\partial}\langle V\rangle.\] Due to $\partial^1g^{(2)}=\partial^2g^{(1)}$ one gets easily the equality $N_2^{\partial}\langle V\rangle^{(1)}=N_1^{\partial}\langle V\rangle^{(2)}$.

Let equalities \[\partial^1N_2^{\partial}\langle V\rangle + N_2^{\partial}\langle V\rangle N_1^{\partial}\langle V\rangle =\partial^2N_1^{\partial}\langle V\rangle + N_1^{\partial}\langle V\rangle N_2^{\partial}\langle V\rangle,\ \ N_2^{\partial}\langle V\rangle^{(1)}=N_1^{\partial}\langle V\rangle^{(2)}\] be true and $g_1\in GL(2,F)$ stand for any particular solution of
$\partial^1gg^{-1}=N_1^{\partial}\langle V\rangle$. It is clear that $g=g_1g_2$ also satisfies it whenever $g_2\in GL(2,F)$ and $\partial^1g_2=0$. There exists such $g_2$ for which $g=g_1g_2$ satisfies the equation $\partial^2gg^{-1}=N_2^{\partial}\langle V\rangle$. Indeed present $\partial^2(g_1g_2)(g_1g_2)^{-1}=N_2^{\partial}\langle V\rangle$ in the form
$\partial^2g_2g_2^{-1}=g_1^{-1}N_2^{\partial}\langle V\rangle g_1-g_1^{-1}\partial^2g_1$. Let us show that its right side is a constant with respect to the $\partial^1$.
\[\partial^1(g_1^{-1}N_2^{\partial}\langle V\rangle g_1-g_1^{-1}\partial^2g_1)=\] \[-g_1^{-1}\partial^1g_1g_1^{-1}(N_2^{\partial}\langle V\rangle g_1-\partial^2g_1)+g_1^{-1}(\partial^1N_2^{\partial}\langle V\rangle g_1 + N_2^{\partial}\langle V\rangle\partial^1g_1-\partial^1\partial^2g_1)=\] \[g_1^{-1}(-N_1^{\partial}\langle V\rangle N_2^{\partial}\langle V\rangle +N_1^{\partial}\langle V\rangle\partial^2g_1g_1^{-1}+\partial^1N_2^{\partial}\langle V\rangle + N_2^{\partial}\langle V\rangle N_1^{\partial}\langle V\rangle -\partial^2\partial^1g_1g_1^{-1})g_1=\] \[g_1^{-1}(-N_1^{\partial}\langle V\rangle N_2^{\partial}\langle V\rangle +N_1^{\partial}\langle V\rangle\partial^2g_1g_1^{-1}+\partial^1N_2^{\partial}\langle V\rangle + N_2^{\partial}\langle V\rangle N_1^{\partial}\langle V\rangle -\] \[\partial^2N_1^{\partial}\langle V\rangle-N_1^{\partial}\langle V\rangle\partial^2g_1g_1^{-1})g_1=0.\]
So system (3.4) has solution $g=g_1g_2\in GL(2,F)$. The equality $N_2^{\partial}\langle V\rangle^{(1)}=N_1^{\partial}\langle V\rangle^{(2)}$ means that such solution belongs to $GL^{\partial}(2,F)$.\end{proof}

Corollary 3.2 and Proposition 3.3 show that for a considered second order LPDE its equivalence to LPDE with constant constants is reduced to the checking of equalities
\[\partial^1N_2^{\partial}\langle V\rangle + N_2^{\partial}\langle V\rangle N_1^{\partial}\langle V\rangle =\partial^2N_1^{\partial}\langle V\rangle + N_1^{\partial}\langle V\rangle N_2^{\partial}\langle V\rangle, N_2^{\partial}\langle V\rangle^{(1)}=N_1^{\partial}\langle V\rangle^{(2)},\] \[N_1^{\partial}\langle V\rangle m^{\partial}\langle V\rangle = \partial^1m^{\partial}\langle V\rangle\ \mbox{and}\ N_2^{\partial}\langle V\rangle m^{\partial}\langle V\rangle= \partial^2m^{\partial}\langle V\rangle.\]
We are refraining from writing the exact expressions for them because of their big sizes. Note also that $g_1,g_2\in GL^{\partial}(2,F)$ are solutions of the same (3.4) if and only if $g_1^{-1}g_2\in GL(2,C)$.

An analog of the above Corollary can be obtained, without any construction of the corresponding map $P^{\partial}\langle  \mathbf{x}\rangle$, in $\gamma_0^{\partial}\langle V\rangle= 0$ case also.

\begin{corollary} Non-parabolic equation (1.1), for which $\partial^2\frac{a^{\partial}\langle  V\rangle}{D\langle V\rangle}=\partial^1\frac{b^{\partial}\langle  V\rangle}{D\langle V\rangle}$ and $\gamma_0^{\partial}\langle V\rangle= 0$, is $G$-equivalent to a LPDE with constant coefficients if and only if the system
\begin{eqnarray} \begin{array}{l}(\partial^igg^{-1})^tM^{\partial}\langle V\rangle + M^{\partial}\langle V\rangle\partial^igg^{-1} + \partial^iM^{\partial}\langle V\rangle-\\ (\Delta^{-1}\partial^i\Delta+\frac{1}{2}D\langle V\rangle^{-1}\partial^iD\langle V\rangle)M^{\partial}\langle V\rangle=0, i=1,2,\\
\partial^igg^{-1}m^{\partial}\langle V\rangle = \partial^im^{\partial}\langle V\rangle, i=1,2 \end{array} \end{eqnarray} has a solution in $GL^{\partial}(2,F)$,
where
$M^{\partial}\langle V\rangle= \left( \begin{array}{cc}
A & B/2  \\
B/2 & C
\end{array} \right),\\
m^{\partial}\langle V\rangle=\left( \left( \begin{array}{c} \frac{a^{\partial}\langle  V\rangle}{D\langle V\rangle} \\
\frac{b^{\partial}\langle  V\rangle}{D\langle V\rangle}\end{array} \right)+ \Delta^{-1}\partial\Delta+\frac{1}{2}D\langle V\rangle^{-1}\partial D\langle V\rangle\right)$.
\end{corollary}

\begin{proof} Indeed the relation  $V_1=\tau^{\partial}\langle h,g; V\rangle\in C^6$ is equivalent to \[ (A_1,B_1,C_1,\frac{a^{\delta}\langle  V_1\rangle}{D\langle V_1\rangle},\frac{b^{\delta}\langle  V_1\rangle}{D\langle V_1\rangle},\gamma_0^{\delta}\langle V_1\rangle)\in C^6.\] Therefore due to Theorem 3.1 and $D\langle V_1\rangle=h^2\Delta^2D\langle V\rangle\in C$ it is equivalent to\\
$h^{-1}\partial h=-\Delta^{-1}\partial\Delta - \frac{1}{2}D\langle V\rangle\partial D\langle V\rangle$ and
\[ \left\{ \begin{array}{l}\partial^i(hg^tM^{\partial}\langle V\rangle g)=0, i=1,2 \\
\partial^i(g^{-1}m^{\partial}\langle V\rangle) = 0, i=1,2 \end{array} \right..\] which can be represented in form of (3.5).\end{proof}

Note that in this case the first two equations of (3.5) are still linear relative to the entries of $\partial^1gg^{-1}, \partial^2gg^{-1}$, now with a singular main matrix, but the next two equations of it are not linear with respect to them as far as \[\Delta^{-1}\partial^1\Delta=  \frac{1}{\Delta}|\begin{array}{cc}
\xi_{xx} & \eta_{xx} \\
\xi_y & \eta_y
\end{array} |+  \frac{1}{\Delta}|\begin{array}{cc}\xi_{x} & \eta_{x} \\
\xi_{xy} & \eta_{xy}\end{array} |,\] \[ \Delta^{-1}\partial^2\Delta= \frac{1}{\Delta}|\begin{array}{cc}
\xi_{xy} & \eta_{xy} \\ \xi_y & \eta_y\end{array} |+\frac{1}{\Delta}|\begin{array}{cc}
\xi_{x} & \eta_{x} \\
\xi_{yy} & \eta_{yy}
\end{array} |.\]

 Of course when the corresponding map $P^{\partial}$ is known one can use the general criterion Theorem 2.3 to get criterion of reducibility for second order non-parabolic LPDE instead of Corollaries 3.2-3.5. Now we are going to consider construction of $P^{\partial}$ for different $G$-invariant subsets of second order non-parabolic LPDE. For this purpose we need some particular cases of Theorem 3.1 in more convenient form.

 \begin{theorem} If $ D\langle V\rangle c^{\partial}\langle  V\rangle \neq 0$  then (1.3) is equivalent to the following system of equalities.
\begin{eqnarray}
\left\{ \begin{array}{l} \left( \begin{array}{cc}
A_{1} & B_{1}/2  \\
B_{1}/2 & C_{1}
\end{array} \right) =h
g^t \left(\begin{array}{cc}
A & B/2  \\
B/2 & C
\end{array} \right) g, \\
\left( \begin{array}{l} \alpha^{\delta}\langle V_{1}\rangle    \\
\beta^{\delta}\langle V_{1}\rangle \end{array} \right)=
g^{-1}\left( \begin{array}{l} \alpha^{\partial}\langle  V\rangle
\\ \beta^{\partial}\langle  V\rangle \end{array} \right),\\
\gamma^{\delta}\langle V_1\rangle=h\gamma^{\partial}\langle V\rangle
\end{array} \right.,
\end{eqnarray}
 where
$\left( \begin{array}{l} \alpha^{\partial}\langle  V\rangle
\\ \beta^{\partial}\langle  V\rangle \end{array} \right)=2\left(\begin{array}{l}\frac{a^{\partial}\langle  V\rangle}{D\langle V\rangle}\\ \frac{b^{\partial}\langle  V\rangle}{D\langle V\rangle}\end{array}\right)+
\frac{1}{D\langle  V\rangle}\partial D\langle V\rangle + 2D_0^{\partial}\langle  V\rangle^{-1}\partial D_0^{\partial}\langle  \mathbf{x}\rangle$, $D_0^{\partial}\langle  \mathbf{x}\rangle=\partial^2\frac{a^{\partial}\langle  \mathbf{x}\rangle}{D\langle \mathbf{x}\rangle}- \partial^1\frac{b^{\partial}\langle  \mathbf{x}\rangle}{D\langle \mathbf{x}\rangle}$. \end{theorem}

\begin{proof} Due to the second equality of (3.1) one has
\[g\left( \begin{array}{l}\frac{a^{\delta}\langle  V_1\rangle}{D\langle V_1\rangle}\\ \frac{b^{\delta}\langle  V_1\rangle}{D\langle V_1\rangle}\\ \end{array}\right)-
\left( \begin{array}{l}\frac{a^{\partial}\langle  V\rangle}{D\langle V\rangle}\\ \frac{b^{\partial}\langle  V\rangle}{D\langle V\rangle}\\ \end{array}\right)= -\left( \begin{array}{l}\frac{h_x}{h}\\ \frac{h_y}{h}\\ \end{array}\right).\]

The equality $(\frac{h_x}{h})_y=(\frac{h_y}{h})_x$ implies that
\[\Delta(\delta_2\frac{a^{\delta}\langle  V_1\rangle}{D\langle V_1\rangle}- \delta_1\frac{b^{\delta}\langle  V_1\rangle}{D\langle V_1\rangle})=\partial^2\frac{a^{\partial}\langle  V\rangle}{D\langle V\rangle}- \partial^1\frac{b^{\partial}\langle  V\rangle}{D\langle V\rangle},\] that is the function
$D_0^{\partial}\langle  \mathbf{x}\rangle=\partial^2\frac{a^{\partial}\langle  \mathbf{x}\rangle}{D\langle \mathbf{x}\rangle}- \partial^1\frac{b^{\partial}\langle  \mathbf{x}\rangle}{D\langle \mathbf{x}\rangle}$ is a relative invariant: \[D_0^{g^{-1}\partial}\langle \tau^{\partial}\langle h,g; \mathbf{x}\rangle\rangle=
\Delta^{-1}D_0^{\partial}\langle  \mathbf{x}\rangle\ \ \mbox{ for any}\ \ (h,g)\in G.\]
Now, in particular, one can see that the function $c^{\partial}\langle  \mathbf{x}\rangle =(D_0^{\partial}\langle  \mathbf{x}\rangle)^2D\langle  \mathbf{x}\rangle$ is also a relative invariant with
property $c^{g^{-1}\partial}\langle \tau^{\partial}\langle h,g; \mathbf{x}\rangle\rangle=
h^{2}c^{\partial}\langle  \mathbf{x}\rangle$ as far as the discriminant $D\langle  \mathbf{x}\rangle$ has property $D\langle \tau^{\partial}\langle h,g; \mathbf{x}\rangle\rangle=
h^2\Delta^{2}D\langle  \mathbf{x}\rangle.$

Due to the last equality one has
\[\frac{1}{D\langle  V_1\rangle}\delta D\langle V_1\rangle=g^{-1}(\frac{1}{D\langle  V\rangle}\partial D\langle V\rangle+2\frac{1}{h}\partial h +2\frac{1}{\Delta}\partial \Delta).\]
Multiply the second equality of (3.1) by 2 and add it to the above equality to get \[2\left(\begin{array}{l}\frac{a^{\delta}\langle  V_1\rangle}{D\langle V_1\rangle}\\ \frac{b^{\delta}\langle  V_1\rangle}{D\langle V_1\rangle}\end{array}\right)+
\frac{1}{D\langle  V_1\rangle}\delta D\langle V_1\rangle =g^{-1}(2\left(\begin{array}{l}\frac{a^{\partial}\langle  V\rangle}{D\langle V\rangle}\\ \frac{b^{\partial}\langle  V\rangle}{D\langle V\rangle}\end{array}\right)+
\frac{1}{D\langle  V\rangle}\partial D\langle V\rangle +2\frac{1}{\Delta}\partial\Delta).\] Substitution $\Delta=
\frac{D_0^{\partial}\langle  V\rangle}{D_0^{\delta}\langle  V_1\rangle}$ into it results in the second equality of (3.6).\end{proof}

We use this result to construct map $P^{\partial}$ for the following cases.

a) $W_e= F^6$- the space of all LPDE of order $\leq 2$ in two variables.

If $D\langle V\rangle c^{\partial}\langle  V\rangle \gamma^{\partial}\langle V\rangle\neq 0$ then
the first and second equalities of (3.6) imply that the following function \[\chi^{\partial}\langle  \mathbf{x}\rangle=\frac{1}{\gamma^{\partial}\langle \mathbf{x}\rangle}
\left( \begin{array}{l} \alpha^{\partial}\langle  \mathbf{x}\rangle
\\ \beta^{\partial}\langle  \mathbf{x}\rangle \end{array} \right)^t \left(\begin{array}{cc}
x_1 & x_2/2  \\
x_2/2 & x_3
\end{array} \right)\left( \begin{array}{l} \alpha^{\partial}\langle  \mathbf{x}\rangle
\\ \beta^{\partial}\langle  \mathbf{x}\rangle \end{array} \right).\]
is a $F^*\times GL^{\partial}(2,F)$-
invariant function.
Therefore the column $\partial\chi^{\partial}\langle  V\rangle$ has property $\delta\chi^{\delta}\langle  V_1\rangle =g^{-1}\partial\chi^{\partial}\langle  V\rangle$.
It implies that the second order matrix  $P_1^{\partial}\langle  \mathbf{x}\rangle$ consisting of columns $\left( \begin{array}{l} \alpha^{\partial}\langle  \mathbf{x}\rangle
\\ \beta^{\partial}\langle  \mathbf{x}\rangle \end{array} \right)$, $\partial\chi^{\partial}\langle  \mathbf{x}\rangle$ has the property
\[P_1^{g^{-1}\partial}\langle \tau^{\partial}\langle h, g; \mathbf{x}\rangle\rangle=g^{-1}P_1^{\partial}\langle  \mathbf{x}\rangle,\] whenever $(h, g)\in F^*\times GL^{\partial}(2,F)$.
For $P_0^{\partial}\langle  \mathbf{x}\rangle$ one can take $\gamma^{\partial}\langle \mathbf{\mathbf{x}}\rangle^{-1}$ and consider map $P^{\partial}\langle  \mathbf{x}\rangle=(P_0^{\partial}\langle  \mathbf{x}\rangle,P_1^{\partial}\langle  \mathbf{x}\rangle)$ which is defined on  \[W_{0e}=\{V\in W_e: P^{\partial}\langle  \mathbf{x}\rangle\ \mbox{ is defined at}\ V\ \mbox{and}\
\gamma^{\partial}\langle V\rangle\det P_1^{\partial}\langle  V\rangle\neq 0\}.\]

As we have noted before now one can reformulate all Theorems 2.2--2.5 for the case a) automatically. The corresponding results, except for the analogue of Theorem 2.3, are presented in [2].
Note that in the case of a) the $\delta$-transcendence degree of the field $C\langle W_e;\partial\rangle^G$ over $C$ equals 3 due to Theorem 2.6.

b) Due to Theorem 3.1 set $W_e$ of $V\in F^6$ where $\gamma^{\partial}\langle \mathbf{x}\rangle$ vanishes is a $G$-invariant subset of $F^6$. The results of Section 2 can be used to describe the set of all $G$-invariant functions on this set as well. To do so one can represent the function $\gamma^{\partial}\langle \mathbf{x}\rangle$ as an irreducible ratio of two $\partial$-differential polynomials. It is not difficult to see that its  numerator is an irreducible polynomial. Let
 now $W_e$ stand for the set of all $V\in F^6$ where its numerator vanishes.

From the first and second equalities of (3.6) one gets the following $GL(n,F)\times GL^{\partial}(2,F)$-relative
invariant function \[\chi_1^{\partial}\langle \mathbf{x}\rangle=
\left( \begin{array}{l} \alpha^{\partial}\langle  \mathbf{x}\rangle
\\ \beta^{\partial}\langle  \mathbf{x}\rangle \end{array} \right)^t \left(\begin{array}{cc}
x_1 & x_2/2  \\
x_2/2 & x_3
\end{array} \right)\left( \begin{array}{l} \alpha^{\partial}\langle  \mathbf{x}\rangle
\\ \beta^{\partial}\langle  \mathbf{x}\rangle \end{array} \right).\] One can use $\chi_1^{\partial}\langle  \mathbf{x}\rangle^{-1}$  for $P_0^{\partial}\langle  \mathbf{x}\rangle$ as far as it has the property
 $\chi_1^{\partial}\langle  V_1\rangle=h\chi_1^{\partial}\langle  V\rangle$ for any $V\in \{ V: D\langle V\rangle c^{\partial}\langle  V\rangle \neq 0, \gamma^{\partial}\langle V\rangle= 0\}$

Due to the second equality of (3.6) for the row        \[(\alpha_1^{\partial}\langle  V\rangle
,\beta_1^{\partial}\langle  V\rangle)=\frac{1}{\chi_1^{\partial}\langle  V\rangle}(\alpha^{\partial}\langle  V\rangle
,\beta^{\partial}\langle  V\rangle)\left(\begin{array}{cc}
A & B/2  \\
B/2 & C
\end{array} \right)\] the equality \[(\alpha_1^{\delta}\langle  V_1\rangle
, \beta_1^{\delta}\langle  V_1\rangle)=(\alpha_1^{\partial}\langle  V\rangle
, \beta_1^{\partial}\langle  V\rangle )g\]
 holds true. The last equality can be presented in the following form
\[\left( \begin{array}{l}\beta_1^{\delta}\langle  V_1\rangle\\ -\alpha_1^{\delta}\langle  V_1\rangle
\\  \end{array}\right)=\Delta g^{-1}\left( \begin{array}{l}\beta_1^{\partial}\langle  V\rangle\\ -\alpha_1^{\partial}\langle  V\rangle\\ \end{array}\right)
\] So for $P_1^{\partial}\langle  \mathbf{x}\rangle$ one can take the matrix consisting of columns
$\left( \begin{array}{l} \alpha^{\partial}\langle  \mathbf{x}\rangle
\\ \beta^{\partial}\langle  \mathbf{x}\rangle \end{array} \right)$ and \[(\partial^2\frac{a^{\partial}\langle  \mathbf{x}\rangle}{D\langle \mathbf{x}\rangle}- \partial^1\frac{b^{\partial}\langle  \mathbf{x}\rangle}{D\langle \mathbf{x}\rangle})\left( \begin{array}{l}\beta_1^{\partial}\langle  \mathbf{x}\rangle\\ -\alpha_1^{\partial}\langle  \mathbf{x}\rangle\\ \end{array}\right).\]

Now one can consider map $P^{\partial}\langle  \mathbf{x}\rangle=(P_0^{\partial}\langle  \mathbf{x}\rangle,P_1^{\partial}\langle  \mathbf{x}\rangle)$ defined on  \[W_{0e}=\{V\in W_e: P^{\partial}\langle  \mathbf{x}\rangle\ \mbox{ is defined at}\ V\ \mbox{and}\
\det P_1^{\partial}\langle  V\rangle\neq 0\}.\]

In this case due to Theorem 2.6 the $\delta$-transcendence degree of the field $C\langle W_e;\partial\rangle^G$ over $C$ is 2 as far as $\dim(W_e)=5$.

c) In its turn the set of $V\in F^6$ for which  $\chi_1^{\partial}\langle  V\rangle=0$ is a $G$-invariant set. Let $W_e$ stand for the numerators's zeros of $\chi_1^{\partial}\langle  \mathbf{x}\rangle$. It is evident that $\dim W_e=4$. For this case $P^{\partial}$ can be constructed in the following way.

Due to Theorem 3.1 for the row \[(\alpha_2^{\partial}\langle  V\rangle
,\beta_2^{\partial}\langle  V\rangle)=(\alpha^{\partial}\langle  V\rangle
,\beta^{\partial}\langle  V\rangle)\left(\begin{array}{cc}
A & B/2  \\
B/2 & C
\end{array} \right)\] the equality \[(\alpha_2^{\delta}\langle  V_1\rangle
, \beta_2^{\delta}\langle  V_1\rangle)=h(\alpha_2^{\partial}\langle  V\rangle
, \beta_2^{\partial}\langle  V\rangle )g\]
 holds true. The last equality can be presented in the following form
\[\left( \begin{array}{l}\beta_2^{\delta}\langle  V_1\rangle\\ -\alpha_2^{\delta}\langle  V_1\rangle
\\  \end{array}\right)=h\Delta g^{-1}\left( \begin{array}{l}\beta_2^{\partial}\langle  V\rangle\\ -\alpha_2^{\partial}\langle  V\rangle\\ \end{array}\right)
\] which implies that

\[\chi_2^{\partial}\langle  V\rangle=
\left( \begin{array}{l} \beta_2^{\partial}\langle  V\rangle\\ -\alpha_2^{\partial}\langle  V\rangle
\\  \end{array} \right)^t \left(\begin{array}{cc}
A & B/2  \\
B/2 & C
\end{array} \right)\left( \begin{array}{l} \beta_2^{\partial}\langle  V\rangle\\ -\alpha_2^{\partial}\langle  V\rangle
\\  \end{array} \right),\] has property $\chi_2^{\delta}\langle  V_1\rangle=h^3\Delta^2\chi_2^{\partial}\langle  V\rangle$ that is $\frac{\chi_2^{\delta}\langle  V_1\rangle}{D\langle  V_1\rangle}=h\frac{\chi_2^{\partial}\langle  V\rangle}{D\langle  V\rangle}.$ Therefore one can take $\frac{D\langle  \mathbf{x}\rangle}{\chi_2^{\partial}\langle  \mathbf{x}\rangle}$ for $P_0^{\partial}\langle  \mathbf{x}\rangle$ and for $P_1^{\partial}\langle \mathbf{x}\rangle$ one can take the matrix consisting of columns
\[\left( \begin{array}{l} \alpha^{\partial}\langle  \mathbf{x}\rangle\\
 \beta^{\partial}\langle  \mathbf{x}\rangle \end{array} \right)\ \ \mbox{and}\  \ \frac{D\langle \mathbf{x}\rangle(\partial^2\frac{a^{\partial}\langle  \mathbf{x}\rangle}{D\langle \mathbf{x}\rangle}- \partial^1\frac{b^{\partial}\langle  \mathbf{x}\rangle}{D\langle \mathbf{x}\rangle})}{\chi_2^{\partial}\langle  \mathbf{x}\rangle}\left( \begin{array}{l}\beta_1^{\partial}\langle \mathbf{x}\rangle\\ -\alpha_1^{\partial}\langle  \mathbf{x}\rangle \end{array}\right).\]

One can consider $ c^{\partial}\langle  V\rangle =0$ case as well. It is not difficult to see that in this case Theorem 2.6 can be presented in the following form.

\begin{theorem} If $ D\langle V\rangle\gamma_0^{\partial}\langle V\rangle \neq 0$ and $ c^{\partial}\langle  V\rangle =0$ then (1.3) is equivalent to the following system of equalities.
\begin{eqnarray}
\left\{ \begin{array}{l} \left( \begin{array}{cc}
A_{1} & B_{1}/2  \\
B_{1}/2 & C_{1}
\end{array} \right) =h
g^t \left(\begin{array}{cc}
A & B/2  \\
B/2 & C
\end{array} \right) g, \\
\left( \begin{array}{l} \alpha_0^{\delta}\langle V_{1}\rangle    \\
\beta_0^{\delta}\langle V_{1}\rangle \end{array} \right)=
g^{-1}\left( \begin{array}{l} \alpha_0^{\partial}\langle  V\rangle
\\ \beta_0^{\partial}\langle  V\rangle \end{array} \right),\\
\gamma_0^{\delta}\langle V_1\rangle=h\gamma_0^{\partial}\langle V\rangle
,\end{array}\right.
\end{eqnarray}
 where
$\left( \begin{array}{l} \alpha_0^{\partial}\langle  V\rangle
\\ \beta_0^{\partial}\langle  V\rangle \end{array} \right)=\left(\begin{array}{l}\frac{a^{\partial}\langle  V\rangle}{D\langle V\rangle}\\ \frac{b^{\partial}\langle  V\rangle}{D\langle V\rangle}\end{array}\right)+
\frac{1}{\gamma_0^{\partial}\langle V\rangle}\partial \gamma_0^{\partial}\langle V\rangle.$\end{theorem}

d) Now consider the set $W_e$ of all $V\in F^6$ where the numerator of $ c^{\partial}\langle  \mathbf{x}\rangle $
vanishes. The dimension of it is 5.

 From the first and second equalities of (3.7) one gets the following $G$-
invariant function \[\chi_0^{\partial}\langle  \mathbf{x}\rangle=\frac{1}{\gamma_0^{\partial}\langle \mathbf{x}\rangle}
\left( \begin{array}{l} \alpha_0^{\partial}\langle  \mathbf{x}\rangle
\\ \beta_0^{\partial}\langle  \mathbf{x}\rangle \end{array} \right)^t \left(\begin{array}{cc}
x_1 & x_2/2  \\
x_2/2 & x_3
\end{array} \right)\left( \begin{array}{l} \alpha_0^{\partial}\langle  \mathbf{x}\rangle
\\ \beta_0^{\partial}\langle  \mathbf{x}\rangle \end{array} \right).\]

Therefore $\partial\chi_0^{\partial}\langle  \mathbf{x}\rangle$ has property $\delta\chi_0^{\delta}\langle  V_1\rangle =g^{-1}\partial\chi_0^{\partial}\langle  V\rangle$.
It implies that second order matrix  $P_1^{\partial}\langle  \mathbf{x}\rangle$ consisting of columns $\left( \begin{array}{l} \alpha_0^{\partial}\langle  \mathbf{x}\rangle
\\ \beta_0^{\partial}\langle  \mathbf{x}\rangle \end{array} \right)$, $\partial\chi_0^{\partial}\langle  \mathbf{x}\rangle$ has the property
\[P_1^{g^{-1}\partial}\langle \tau^{\partial}\langle h, g; \mathbf{x}\rangle\rangle=g^{-1}P_1^{\partial}\langle  \mathbf{x}\rangle,\] whenever $(h, g)\in F^*\times GL^{\partial}(2,F)$.
For the $P_0^{\partial}\langle  \mathbf{x}\rangle$ one can take $\gamma_0^{\partial}\langle \mathbf{x}\rangle^{-1}$ and for the domain of the corresponding $P^{\partial}$ the set
\[W_{0e}=\{ V\in W_e: P^{\partial}\langle  \mathbf{x}\rangle\ \mbox{ is defined at}\ V\ \mbox{and}\
\det P_1^{\partial}\langle  V\rangle\neq 0\}.\] In this case $\delta-tr.deg. C\langle W_e;\partial\rangle^G/C= 2.$

e) The last case of non-parabolic LPDE we want to consider is $ c^{\partial}\langle  V\rangle =\gamma_0^{\partial}\langle V\rangle= 0$ case.
First of all due to $ c^{\partial}\langle  V\rangle = 0$ there exists $h^0$ such that $(h^0)^{-1}\partial h^0=\left( \begin{array}{l}\frac{a^{\partial}\langle  V\rangle}{D\langle V\rangle}\\ \frac{b^{\partial}\langle  V\rangle}{D\langle V\rangle}\\ \end{array}\right)$. Assume $A\neq 0$ and consider $g^{0}=\left(\begin{array}{cc}
\lambda_1g^0_{21} & \lambda_2g^0_{22}  \\
g^0_{21} & g^0_{22}
\end{array} \right)$, $\partial^0=(g^0)^{-1}\partial$ for which $\partial^2(\lambda_ig^0_{2i})=\partial^1g^0_{2i}$, $i=1,2$ holds true, where $\lambda_1, \lambda_2$ are roots of the equation $A\lambda^2+B\lambda+C=0.$

One has $(0,B^0,0,0,0,0)=\tau^{\partial}\langle h^0, g^0; V\rangle$. In other words such equations are equivalent (reducible) to $Bu_{xy}=0$ type equations. It is not difficult to see that the set $W_e=\{V=(0,B,0,0,0,0): B\in F\}$ is invariant with respect to $(h,g)\in F^*\times GL^{\partial}(2, F)$ if and only if $h\in C^*$ and all main or bias diagonal elements of $g$ are zeros. Therefore for simplicity we describe here the field of invariants of $Bu_{xy}=0$ type equations with respect to $(h,g)\in F^*\times GL^{\partial}(2, F)$ for which $h\in C^*$ and all main or bias diagonal elements of $g$ are zeros. It is easy to check that function $f^{\partial}\langle B\rangle=B^{-2}(B^{-1}B_x)_y$ has property $f^{\delta}\langle B_1\rangle=h^2f^{\partial}\langle B\rangle$ and therefore for $f_1^{\partial}\langle B_1\rangle=(\partial f^{\partial}\langle B\rangle)^t\left( \begin{array}{cc}
0 & B/2  \\
B/2 & 0\end{array} \right)\partial f^{\partial}\langle B\rangle$ the equality $f_1^{\delta}\langle B_1\rangle=h^5f_1^{\partial}\langle B\rangle$ holds true. So for $P_0^{\partial}\langle B\rangle$ one can take $\frac{f^{\partial}\langle B\rangle^2}{f_1^{\partial}\langle B\rangle}$. For $P_1^{\partial}\langle B\rangle$ the matrix with columns  $f^{\partial}\langle B\rangle^{-1}\partial f^{\partial}\langle B\rangle$ and $f_1^{\partial}\langle B\rangle^{-1}\partial f_1^{\partial}\langle B\rangle$ can be taken. In this case the $\delta$-transcendence degree of the corresponding invariant rational functions over $C$ is 1.

{\bf Conclusion.} The paper deals with classification and invariants problems of LPDE which have not been much explored before. It offers a general method to classification and invariants problems of such equations. An constructive application of the method is presented in the case of second order non-parabolic LPDEs. A criterion for reducibility of such equations to LPDE with constant coefficients is offered.

\end{document}